\documentclass[12pt, reqno, a4paper]{amsart}
\usepackage{times}

\usepackage[utf8]{inputenc}
\usepackage[USenglish]{babel}
\usepackage{amsmath,amsthm,amssymb,amsfonts}
\usepackage{mathtools}
\usepackage{mathrsfs}
\usepackage{bbm}

\usepackage{indentfirst}
\usepackage{enumitem}
\usepackage{xcolor}

\usepackage{cancel} 
\usepackage{xifthen}

\usepackage{float}
\usepackage{placeins}
\usepackage{caption}

\usepackage{tikz}
\usetikzlibrary{patterns}
\usepackage{tikz-cd}

\usetikzlibrary{decorations.pathreplacing,decorations.markings}
\usetikzlibrary{babel}

\usepackage[breaklinks=true, bookmarksopenlevel=1, bookmarksdepth=2]{hyperref}
\hypersetup{
colorlinks,
   linkcolor={cyan!80!black},
   citecolor={cyan!80!black},
 urlcolor={cyan!80!black}
}

\allowdisplaybreaks

\newtheorem{thm}{Theorem}[section]

\newtheorem{lem}[thm]{Lemma}

\theoremstyle{plain} 
\newcommand{\thistheoremname}{}
\newtheorem*{genericthm}{\thistheoremname}

\theoremstyle{definition}

\theoremstyle{remark}

\newtheorem*{xrem}{Remark}

\numberwithin{equation}{section}

\newcommand{\Z}{\mathbb{Z}}      
\newcommand{\R}{\mathbb{R}}      

\renewcommand{\pmod}[1]{         
  ~(\mathrm{mod}~#1)}

\newcommand{\E}{\mathbb{E}}   

\newcommand{\A}{\mathscr{A}}

\usepackage[margin=1in]{geometry}

\restylefloat{table}
\setlist[itemize]{leftmargin=*}
\setlist[enumerate]{leftmargin=*}

\begin{document}

\title{A note on the Cram\'er--Granville model}%
\author{Christian T\'afula}%
\address{D\'epartement de math\'ematiques et de statistique\\
Universit\'e de Montr\'eal\\
CP 6128 succ. Centre-Ville\\
Montr\'eal, QC H3C 3J7\\
Canada}%
\email{christian.tafula.santos@umontreal.ca}%


 \begin{abstract}
  We show the existence of a set $A\subseteq \mathbb{Z}_{\geq 2}$ satisfying the estimates of the Bateman--Horn conjecture, Goldbach's conjecture, and also
  \[ \#\{p\leq x \text{ prime} ~|~ p\in A\} \gg x(\log\log x)/(\log x)^2. \]
 \end{abstract}
\maketitle

\section{Introduction}
 The Bateman--Horn conjecture, a quantitative form of Schinzel's hypothesis H (cf. Schinzel--Sierpinski \cite{schsie58}), may be stated as follows: Say a family $f_1$, $\ldots$, $f_{k} \in \Z[x]$ of distinct irreducible polynomials with positive leading coefficient is \emph{admissible} if $f := f_1\cdots f_k$ has no fixed divisor $>1$. That is, writing
 \[ \omega_f(p) := \#\{ 0\leq \ell < p ~|~  f(\ell) \equiv 0\pmod{p}\}, \]
 we suppose that $\omega_{f}(p) < p$ for every prime $p$. Then, we expect that
 \begin{equation}
  \#\{n\leq x ~|~ f_1(n),\ldots,f_{k}(n) \text{ are prime} \} = \bigg(\frac{C_{f}}{\deg f} + O_f\bigg(\frac{1}{\log x}\bigg)\bigg) \frac{x}{(\log x)^{k}}, \label{BHconj}
 \end{equation}
 where
 \begin{equation}
  C_{f} := \prod_{p} \bigg(1-\frac{1}{p}\bigg)^{-k}\bigg(1-\frac{\omega_f(p)}{p}\bigg). \label{const}
 \end{equation}
 Bateman--Horn \cite{bathor62} proved the convergence of \eqref{const},\footnote{In the sense of truncating the product for $p\leq T$ and then taking $T\to \infty$.} and calculated the particular case $f_1(n) = n$, $f_2(n) = n^2+n+1$ for $n$ up to 113000. 
 The guess at the main in \eqref{BHconj} term comes from the works of Hardy--Littlewood.\footnote{\eqref{BHconj} generalizes Conjectures B, D, E, F, K, P and Theorem XI of \cite{harlit23}.} The structure of the product follows from an application of Selberg's upper bound sieve, and appears in Bateman--Stemmler \cite{batste62}. Lastly, the error term of $O((\log x)^{-1})$ is suggested by the application of the Cram\'er--Granville model we use below (when taking $T=x^{1/2}$ and ignoring certain error terms). 
 
 In comparison with Conjectures C, G, L of Hardy--Littlewood \cite{harlit23}, Schinzel \cite{sch63} conjectured a modified version of \eqref{BHconj} called Hypothesis H\textsubscript{N} by Halberstam--Richert \cite{halberstam11}. The main subcase is Hardy--Littlewood's extended Goldbach's conjecture, which predicts that, for even $N$,
 \begin{equation}
  \#\{p\leq N ~|~ N-p \text{ is prime} \} = 2C_2\prod_{p\mid N} \bigg(\frac{p-1}{p-2}\bigg) \frac{N}{(\log N)^2}\bigg(1 + O\bigg(\frac{\log\log N}{\log N}\bigg) \bigg), \label{golD}
 \end{equation}
 where $C_2 := \prod_{p\geq 3} (1-(p-1)^{-2}) = 0.66016\ldots$ is the twin primes constant. We chose the error term in \eqref{golD} in comparison with Vinogradov's theorem on sums of three primes (cf. Nathanson \cite[Theorem 8.1]{nathanson96}).
 
 In this note, we use the Cram\'er--Granville model combined with the probabilistic method to obtain the following:
 
 \begin{thm}\label{MT1}
  There exists a set $A\subseteq \Z_{\geq 2}$ satisfying the following:
  
  \begin{enumerate}[label=\textnormal{(\roman*)}]
   \item\label{it1} For every admissible family of polynomials $f_1,\ldots,f_{k} \in \Z[x]$, we have
   \[ \#\{n\leq x ~|~ f_1(n),\ldots,f_{k}(n) \in A\} = \bigg(\frac{C_{f}}{\deg f} + O_f\bigg(\frac{1}{\log\log x}\bigg)\bigg) \frac{x}{(\log x)^{k}}; \]
   
   \item\label{it2} For even $N$,
   \[ \#\{ n\leq N ~|~ n,\, N-n \in A \} =  2C_2\prod_{p\mid N} \bigg(\frac{p-1}{p-2}\bigg)\frac{N}{(\log N)^2}\bigg(1 + O\bigg(\frac{\log\log\log N}{\log\log N}\bigg) \bigg); \]
   
   \item We have
   \[ \#\{p\leq x \textnormal{ prime} ~|~ p\in A\} \gtrsim e^{\gamma}\frac{x}{(\log x)^2}\log\log x, \]
   where $\gamma = 0.57721\ldots$ is the Euler--Mascheroni constant.
  \end{enumerate}
 \end{thm}
 
 It would be interesting to find similar sets with a larger density of primes.

\section{Probabilistic model}
 Let $T = T(x) := \log x/\log\log\log x$, and write $P_T := \prod_{p\leq T} p$. Let $\A\in\Z_{\geq 2}$ be a random subset defined by:
 \begin{equation}
  \Pr(n\in \A) = \E(\mathbbm{1}_{\A}(n)) :=
  \begin{cases}
   \displaystyle \prod_{p\leq T(n)} \bigg(1-\frac{1}{p}\bigg)^{-1} \frac{1}{\log n}, &\text{if } (n, P_{T(n)}) = 1 \\
   0,&\text{otherwise.}
  \end{cases} \label{defpr}
 \end{equation}
 The product measure of the Bernoulli random variables $\mathbbm{1}_{\A}(n)$, $n\geq 2$, may be interpreted as a measure on the set of subsets $\A\subseteq \Z_{\geq 2}$. Since the intersection of countably many events of probability $1$ has probability $1$ (and hence is non-empty), for any countable collection of properties proven to hold with probability $1$ with respect to \eqref{defpr}, there exists a set $A\subseteq \Z_{\geq 2}$ satisfying all those properties. 
 
 \begin{xrem}
  The idea for this model as a tool to predict properties of primes dates back to Cram\'er \cite{cra36}, who used the simpler model ``$\Pr(n\in \A) = 1/\log n$'' (or \eqref{defpr} with $T=1$) to conjecture bounds on the size of gaps between consecutive primes. This was later refined by Granville \cite{gra95, gran94} to the form above. With our choice of $T$, as a fixed slow function of $n$, \eqref{defpr} yields a fixed measure on the subsets of $\Z_{\geq 2}$, allowing us to apply the probabilistic method.
 \end{xrem}
 
 \subsection{The Kim--Vu inequality}
  Let $n\in \Z_{\geq 1}$, and take $v_1$, $\ldots$, $v_{n}$ to be independent, not necessarily identically distributed, $\{0,1\}$-random variables. A \emph{boolean polynomial} is a multivariate polynomial
 \[ Y(v_1,\ldots,v_n) = \sum_{i} c_i I_i \in \R[v_1,\ldots,v_{n}], \]
 where the $I_i$s are monomials: products of some of the $v_k$s. We say that $f$ is \emph{positive} if $c_i \in \R_{>0}$ for every $i$, and \emph{simple} if the largest exponent of $v_i$ in a monomial is $1$ for every $i$. For a non-empty multiset\footnote{A \emph{multiset} is a set that allows multiple instances of an element.} $S\subseteq \{v_1,\ldots,v_n\}$, define $\partial_S :=  \prod_{v\in S} \partial_{v}$, where $\partial_{v}$ is the partial derivative in $v$. For example: if $S = \{1,1,2\}$, then $\partial_{S} (v_1^3 v_2 v_3 + 3v_1^5) = 6v_1 v_3$. Define
 \[ \E'(Y) = \max_{\substack{S\subseteq \{v_1,\ldots,v_n\} \\ \text{multiset},\, |S| \geq 1}} \E(\partial_S Y). \]
 We will use the following concentration result:

 \begin{lem}[Kim--Vu \cite{kimvvu00}]\label{kimvu}
  Let $k\geq 1$, and $Y(v_1,\ldots,v_n)$ is a positive, simple boolean polynomial of degree $k$. Write $E' := \E'(Y)$ and $E := \max\{\E(Y), E'\}$. Then, for any real $\lambda \geq 1$, we have
  \[ \Pr\big(|Y-\E(Y)| > 8^{k}\sqrt{k!}\, \lambda^k (E'E)^{1/2} \big) \ll_{k} n^{k-1}e^{-\lambda} \]
 \end{lem}
 
 The quantity $E'$ can be thought of as measuring how correlated are the monomials of $Y$. In applications, we will take $\lambda = (k+1)\log n$ in order to use the Borel--Cantelli lemma, which yields strong concentration when $1\ll E' = o(E/(\log n)^{2k})$. 
 
 \subsection{The Bateman--Horn conjecture}
 Let $f_1$, $\ldots$, $f_{k} \in \Z[x]$ be distinct irreducible polynomials as in the introduction, and $f := f_1\cdots f_k$. For large $n$, we have (say) $f_1(n) < \cdots < f_{k}(n)$ (if not, reorder the $f_i$); in particular, they are all independent. For $U_{+} := \lceil T(f_k(x)) \rceil$, we have
 \begin{align}
  \E&(\#\{n\leq x ~|~ f_1(n),\ldots,f_k(n)\in \A\}) \nonumber \\
  &= \sum_{\sqrt{x}\leq n\leq x} \Pr(f_1(n)\in \A)\cdots \Pr(f_k(n)\in \A) + O(\sqrt{x}) \nonumber \\
  &= \sum_{r=0}^{(x-\sqrt{x})/P_{U_{+}}} \sum_{n=\sqrt{x}+rP_{U_{+}}}^{\sqrt{x}+ (r+1)P_{U_{+}}} \Pr(f_1(n)\in \A)\cdots \Pr(f_k(n)\in \A) + O(P_{U_{+}} + \sqrt{x}) \nonumber \\
  &= \frac{S(x)}{\deg f} \sum_{r=0}^{(x-\sqrt{x})/P_{U_{+}}} \frac{P_{U_{+}}+O(1)}{\log(\sqrt{x}+ rP_{U_{+}})^{k}} + O(4^{U_{+}} + \sqrt{x}), \label{eqsplt}
 \end{align}
 where, by the Chinese remainder theorem, for $U_{-} := \lfloor T(f_1(\sqrt{x}))\rfloor$,
 \begin{equation}
  \prod_{p\leq U_{-}}\bigg(1-\frac{1}{p}\bigg)^{-k} \prod_{p\leq U_{+}} \bigg(1 - \frac{\omega_f(p)}{p}\bigg) \leq S(x) \leq \prod_{p\leq U_{+}}\bigg(1-\frac{1}{p}\bigg)^{-k} \prod_{p\leq U_{-}} \bigg(1 - \frac{\omega_f(p)}{p}\bigg). \label{bddS}
 \end{equation}
 In order to estimate $S(x)$, we need the following lemma:
 
 \begin{lem}\label{mert}
  We have:\smallskip
  \begin{enumerate}[label=\textnormal{(\roman*)}]
   \item $\displaystyle \prod_{p\leq T} \bigg(1-\frac{1}{p}\bigg)^{-k} = e^{k\gamma} (\log T)^k + O((\log T)^{k-1})$;\smallskip
   
   \item $\displaystyle \prod_{p\leq T} \bigg(1-\frac{\omega_f(p)}{p}\bigg) = C_f\,e^{-k\gamma} (\log T)^{-k} + O((\log T)^{-(k+1)})$.
  \end{enumerate}
 \end{lem}
 \begin{proof}
  Part (i) follows from Mertens' third theorem. For part (ii), we sketch the argument following Bateman--Horn \cite{bathor62}. Write $\omega_i(p):= \{0\leq \ell < p ~|~ f_i(\ell)\equiv 0\pmod{p}\}$. Since $\gcd(f_i,f_j) \in \Z$ for every pair $i,j\leq k$, the roots of the $f_i$ will all be distinct modulo $p$ for large $p$, so that $\omega_f(p) = \omega_1(p) + \cdots + \omega_k(p)$.
  
  For all but finitely many $p$, $\omega_i(p)$ is the number of distinct prime ideals of norm $p$ in the number field generated by a root of $f_i$. Thus, by the prime ideal theorem,\footnote{We use that $\#\{\mathfrak{p}\subseteq \mathcal{O}_K ~|~ \mathrm{N}\mathfrak{p} \leq x\} = (1+ O_K(\frac{1}{\log x})) \frac{x}{\log x}$ --- cf. Satz 191 of Landau \cite{landau18}.} one obtains by partial summation that
  \begin{align*}
   \sum_{p\leq T} \frac{\omega_i(p)}{p} &= \log\log T + A_i + O\bigg(\frac{1}{\log T}\bigg) \\
   &= \sum_{p\leq T} \frac{1}{p} + B_i + O\bigg(\frac{1}{\log T}\bigg),
  \end{align*}
  for some constants $A_i$, $B_i$. Hence, there is some constant $D_f$ for which
  \begin{equation}
   \sum_{p\leq T} \frac{\omega_f(p)-k}{p} = D_f + O\bigg(\frac{1}{\log T}\bigg) \label{Dfcns}
  \end{equation}
  
  Using that 
  \[ 1-\frac{\omega_f(p)}{p} = \bigg(1-\frac{1}{p}\bigg)^{k}\bigg(1 - \frac{\omega_f(p)-k}{p} + O\bigg(\frac{1}{p^2}\bigg) \bigg), \] 
  it follows from Mertens' third theorem and \eqref{Dfcns} that
  \begin{align*}
   \prod_{p\leq T} \bigg(1-\frac{\omega_f(p)}{p}\bigg) &= \prod_{p\leq T}\bigg(1-\frac{1}{p}\bigg)^{k} \prod_{p\leq T} \bigg(1 - \frac{\omega_f(p)-k}{p} + O\bigg(\frac{1}{p^2}\bigg) \bigg) \\
   &= \frac{\exp(D_f + G_f + O((\log T)^{-1}))}{e^{k\gamma} (\log T)^{k} + O((\log T)^{k-1})} \\
   &= \bigg(1+ O\bigg(\frac{1}{\log T}\bigg)\bigg) C_f e^{-k\gamma} (\log T)^{-k}.
  \end{align*}
  where $G_f$ is the constant coming from $O(1/p^2)$, and $C_f = e^{D_f+G_f}$.  
 \end{proof}

 Since $\log T(x^{t}) = \log T(x) + O(1)$ uniformly for $\frac{1}{2} \leq t \leq 2\deg f$, it follows from \eqref{eqsplt}, \eqref{bddS} and Lemma \ref{mert} that
 \begin{align}
  \E(\#\{n\leq x ~|~ &f_1(n),\ldots,f_k(n)\in \A\}) \nonumber \\
  &= \bigg(\frac{C_f}{\deg f}+ O\bigg(\frac{1}{\log T}\bigg)\bigg) \int_{\sqrt{x}}^{x} \frac{\mathrm{d}t}{(\log t)^k} + O(\sqrt{x}) \nonumber \\ 
  &=\bigg(\frac{C_f}{\deg f} + O\bigg(\frac{1}{\log\log x}\bigg)\bigg) \frac{x}{(\log x)^k}. \label{expec}
 \end{align}
 In the notation of Kim-Vu's inequality (Lemma \ref{kimvu}),
 \[ Y(x) := \#\{\sqrt{x} \leq n\leq x ~|~ f_1(n),\ldots,f_k(n)\in \A\} = \sum_{\sqrt{x}\leq n\leq x} \mathbbm{1}_{\A}(f_1(n))\cdots \mathbbm{1}_{\A}(f_k(n)) \]
 is a positive, simple boolean polynomial of degree $k$ in $\asymp x$ variables, with $E'\leq k$. Indeed, each $\mathbbm{1}_{\A}(m)$ appears in at most $k$ distinct monomials for large $x$, since the $f_i$ are increasing on $[\sqrt{x}, x]$. Taking $\lambda = (k+1)\log x$, we have 
 \begin{align*}
  \Pr(|Y(x) - \E(Y(x))| \geq 8^k\sqrt{k!}\,\lambda^k (k\, \E(Y(x)))^{1/2}) \ll x^{-2}.
 \end{align*}
 It follows from \eqref{expec} that $8^k\sqrt{k!}\,\lambda^k (k\, \E(Y(x)))^{1/2} = O(\sqrt{x}\,(\log x)^{k/2})$. Thus, applying the Borel--Cantelli lemma (for integer $x$), we conclude that
 \begin{align*}
  \#\{n\leq x ~|~ f_1(n),\ldots,f_k(n)\in \A\} &\stackrel{\textnormal{a.s.}}{=} \E(\#\{n\leq x ~|~ f_1(n),\ldots,f_k(n)\in \A\}) \\
  &\hspace{+12em}+ O(\sqrt{x}\,(\log x)^{k/2}).
 \end{align*}
 Therefore, from \eqref{expec}, there exists $A\subseteq \Z_{\geq 2}$ such that
 \begin{equation*}
  \#\{n\leq x ~|~ f_1(n),\ldots,f_k(n)\in A\} = \bigg(\frac{C_f}{\deg f} + O_f\bigg(\frac{1}{\log\log x}\bigg)\bigg) \frac{x}{(\log x)^k}. 
 \end{equation*}
 for every admissible family of polynomials $f_1$, $\ldots$, $f_k$.

 \subsection{Goldbach's conjecture}
 For even $N \in \Z_{\geq 2}$, write
 \[ r_{\A,2}(N) := \sum_{n\leq N} \mathbbm{1}_{\A}(n) \mathbbm{1}_{\A}(N-n), \]
 so that, for $T = T(N)$,
 \begin{align}
  \E(r_{\A,2}(N)&) = \sum_{\sqrt{N} \leq n\leq N-\sqrt{N}} \Pr(n\in \A) \Pr(N-n \in \A) + O(\sqrt{N}) \nonumber \\
  &= \sum_{r=0}^{(N-\sqrt{N})/P_{T}} \sum_{n=\sqrt{N}+rP_T}^{\sqrt{N}+ (r+1)P_T} \Pr(n\in \A) \Pr(N-n \in \A)+ O(P_T + \sqrt{N}) \nonumber \\
  &= S(N) \sum_{r=0}^{(N-\sqrt{N})/P_{T}} \frac{P_T+O(1)}{\log(\sqrt{N}+ rP_T) \log(N - \sqrt{N} - rP_T)} + O(4^T + \sqrt{N}), \label{eqgold}
 \end{align}
 where, by the Chinese remainder theorem, for $U := T(\sqrt{N})$ ($<T$),
 \begin{align*}
  \prod_{p\leq U}\bigg(1-\frac{1}{p}\bigg)^{-2} \prod_{\substack{p\leq T \\ p\nmid N}} \bigg(1 - \frac{2}{p}\bigg) &\prod_{\substack{p\leq T \\ p\mid N}} \bigg(1 - \frac{1}{p}\bigg) \leq S(N) \\
  &\leq \prod_{p\leq T}\bigg(1-\frac{1}{p}\bigg)^{-2} \prod_{\substack{p\leq U \\ p\nmid N}} \bigg(1 - \frac{2}{p}\bigg) \prod_{\substack{p\leq U \\ p\mid N}} \bigg(1 - \frac{1}{p}\bigg).
 \end{align*}
 Using that
 \[ \bigg(1-\frac{2}{p}\bigg) = \bigg(1-\frac{1}{p}\bigg)^{2}\bigg(1-\frac{1}{(p-1)^2}\bigg), \]
 we get
 \begin{equation*}
  S(N) = 2\prod_{3\leq p\leq T} \bigg(1-\frac{1}{(p-1)^2}\bigg) \prod_{\substack{3\leq p \leq T \\ p\mid N}} \bigg(\frac{p-1}{p-2}\bigg) \cdot R(N),
 \end{equation*}
 where
 \[ \prod_{U\leq p\leq T} \bigg(1-\frac{1}{p} \bigg)^2 \leq R(N) \leq \prod_{\substack{U\leq p\leq T \\ p\nmid N}} \bigg(1 - \frac{2}{p}\bigg)^{-1} \prod_{\substack{U \leq p \leq T \\ p\mid N}} \bigg(1 - \frac{1}{p}\bigg)^{-1} \]
 Since $\log\log T - \log\log U = O(1/\log N)$, it follows that $R(N) = 1 + O(1/\log N)$, so
 \begin{equation}
  S(N) = 2\prod_{3\leq p\leq T} \bigg(1-\frac{1}{(p-1)^2}\bigg) \prod_{\substack{3\leq p \leq T \\ p\mid N}} \bigg(\frac{p-1}{p-2}\bigg) \bigg(1+O\bigg(\frac{1}{\log N}\bigg)\bigg). \label{SNsmfn}
 \end{equation}
 With that we can show the following:
 
 \begin{lem}\label{goldSN}
  We have
  \[ S(N) =  2C_2 \prod_{\substack{p\geq 3 \\ p\mid N}} \bigg(\frac{p-1}{p-2}\bigg) \bigg(1+O\bigg(\frac{\log\log\log N}{\log\log N}\bigg)\bigg), \]
  where $C_2 = \prod_{p\geq 3} (1 - \frac{1}{(p-1)^{2}})$.
 \end{lem}
 \begin{proof}
  From \eqref{SNsmfn}, it suffices to analyse the two products. For the first, we have
  \begin{align*}
   \prod_{3\leq p\leq T} \bigg(1-\frac{1}{(p-1)^2}\bigg) &= \exp\bigg(-\sum_{3\leq p\leq T} \bigg(\frac{1}{(p-1)^2} + O\bigg(\frac{1}{p^3}\bigg)\bigg) \bigg) \\
   &= \exp\bigg(-\sum_{p\geq 3} \bigg(\frac{1}{(p-1)^2} + O\bigg(\frac{1}{p^3}\bigg)\bigg) + O\bigg(\frac{1}{T}\bigg)\bigg) \\
   &= C_2 + O\bigg(\frac{1}{T}\bigg).
  \end{align*}
  For the second product, note that there are at most $O(\frac{\log N}{\log\log N})$ primes $p\mid N$ such that $p\geq T$. Using this, we get 
  \begin{align*}
   \prod_{\substack{3\leq p \leq T \\ p\mid N}} \bigg(\frac{p-1}{p-2}\bigg) &= \prod_{\substack{p\geq 3 \\ p\mid N}} \bigg(\frac{p-1}{p-2}\bigg) \prod_{\substack{p \geq T \\ p\mid N}} \bigg(1 + \frac{1}{p-2}\bigg)^{-1} \\
   &= \prod_{\substack{p\geq 3 \\ p\mid N}} \bigg(\frac{p-1}{p-2}\bigg) \exp\Bigg(\sum_{\substack{p \geq T \\ p\mid N}} \bigg(\frac{1}{p-2} + O\bigg(\frac{1}{p^2}\bigg)\bigg)\Bigg) \\
   &= \prod_{\substack{p\geq 3 \\ p\mid N}} \bigg(\frac{p-1}{p-2}\bigg) \exp\Bigg(O\bigg(\frac{1}{T}\frac{\log N}{\log\log N}\bigg)\Bigg) \\
   &= \prod_{\substack{p\geq 3 \\ p\mid N}} \bigg(\frac{p-1}{p-2}\bigg) \bigg(1+O\bigg(\frac{\log\log\log N}{\log\log N}\bigg)\bigg),
  \end{align*}
  concluding the proof.
 \end{proof}

 Thus, it follows from Lemma \ref{goldSN} and \eqref{eqgold} that
 \begin{align}
  \E&(r_{\A,2}(N)) \nonumber \\
  &=  2C_2 \prod_{\substack{p\geq 3 \\ p\mid N}} \bigg(\frac{p-1}{p-2}\bigg)\bigg(1+O\bigg(\frac{\log\log\log N}{\log\log N}\bigg)\bigg) \int_{\sqrt{N}}^{N} \frac{\mathrm{d}t}{\log(t)\log(N-t)} + O(\sqrt{N}) \nonumber \\ 
  &= 2C_2 \prod_{\substack{p\geq 3 \\ p\mid N}} \bigg(\frac{p-1}{p-2}\bigg) \frac{N}{(\log N)^2} \bigg(1+O\bigg(\frac{\log\log\log N}{\log\log N}\bigg)\bigg). \label{expec2}
 \end{align}
 We can apply Kim--Vu's inequality (Lemma \ref{kimvu}) taking $\lambda = 2\log N$, which, since $E' \leq 2$, yields
 \begin{align*}
  \Pr(|r_{\A,2}(N) - \E(r_{\A,2}(N))| \geq 128\,\lambda^2\, \E(r_{\A,2}(N))^{1/2}) \ll N^{-2}.
 \end{align*}
 It follows from \eqref{expec2} that $\lambda^2\, \E(r_{\A,2}(N))^{1/2} = O(\sqrt{N}\log N)$. Thus, applying the Borel--Cantelli lemma, we conclude that
 \begin{align*}
  \#\{n\leq N ~|~ N-n\in \A\} &\stackrel{\textnormal{a.s.}}{=} \E(\#\{n\leq N ~|~N-n\in \A\}) + O(\sqrt{N}\log N).
 \end{align*}
 Therefore, there is a set $A\subseteq \Z_{\geq 2}$ satisfying Theorem \ref{MT1} (i) and
 \begin{equation*}
  \#\{n\leq N ~|~ N-n\in A\} = 2C_2\prod_{\substack{p\geq 3 \\ p\mid N}} \bigg(\frac{p-1}{p-2}\bigg) \frac{N}{(\log N)^2} \bigg(1+O\bigg(\frac{\log\log\log N}{\log\log N}\bigg)\bigg). 
 \end{equation*}

\subsection{How many primes are in \texorpdfstring{$A$}{A}?}
 Since, by Lemma \ref{mert} (i),
 \begin{align}
  \E(\#\{p\leq x ~|~ p\in \A\}) &= \sum_{\sqrt{x}< n\leq x} \Pr(p\in \A) + O(\sqrt{x}) \nonumber \\
  &= \big(e^{\gamma} \log T + O(1)\big) \int_{\sqrt{x}}^{x} \frac{\mathrm{d}\pi(t)}{\log t} + O(\sqrt{x}) \nonumber \\
  &= e^{\gamma}\frac{x}{(\log x)^2}(\log T + O(1)),\nonumber
 \end{align}
 and the strong law of large numbers\footnote{Alternatively, one can use Kim-Vu's inequality together with the Borel--Cantelli lemma.} implies $\#\{p\leq x ~|~ p\in \A\} \stackrel{\text{a.s.}}{\sim} \E(\#\{p\leq x ~|~ p\in \A\})$, we can choose $A$ satisfying
 \[ \#\{p\leq x ~|~ p\in A\} \gtrsim e^{\gamma}\frac{x}{(\log x)^2}\log\log x. \]

\addtocontents{toc}{\protect\setcounter{tocdepth}{0}}
\section*{Acknowledgements}
 I thank the referee for helpful comments and careful reading.
 
\addtocontents{toc}{\protect\setcounter{tocdepth}{1}}

{
\bibliographystyle{amsplain}
\bibliography{$HOME/Academie/Recherche/_latex/bibliotheca}%
}
\end{document}